\documentclass{amsart}
\usepackage{mathrsfs}
\usepackage{amssymb}
\usepackage{graphicx} 
\usepackage{geometry}
\usepackage{enumerate}

\usepackage{xcolor}

\title{Mean field equations arising from \\ random vortex dynamics}
\newcommand{\R}{\mathbb{R}}
\newcommand{\E}{\mathbb{E}}
\newcommand{\rd}{\mathrm{d}}
\newcommand{\K}{\mathcal{K}}
\newcommand{\tb}{\tilde{b}}

\newtheorem{lemma}{Lemma}
\newtheorem{proposition}{Proposition}
\newtheorem{theorem}{Theorem}
\newtheorem{remark}{Remark}

\begin{document}

 \author{Jiawei Li}
 \address{School of Mathematics, University of Edinburgh, Edinburgh, United Kingdom, EH9 3FD.}
 \email{jiawei.li@ed.ac.uk}

\author{Zhongmin Qian}
\address{Mathematical Institute, University of Oxford, Oxford, United Kingdom, OX2 6GG.}
\email{zhongmin.qian@maths.ox.ac.uk}



\begin{abstract}
We consider Mckean-Vlasov type stochastic differential equations with multiplicative noise arising from the random vortex method. Such an equation can be viewed as the mean-field limit of interacting particle systems with singular interacting kernels such as the Biot-Savart kernel. A new estimate for the transition probability density of diffusion processes will be formulated to handle the singularity of the interacting kernel. The existence and uniqueness of the weak solution of such SDEs will be established as the main result.  
\end{abstract}

\maketitle

\section{Introduction}
In this paper, we aim to establish the existence and uniqueness of the weak solution of the $\R^d$-valued stochastic differential equation  
\begin{equation}
\label{stochastic-vortex-SDE}
    \rd X_t = u(t,X_t) \rd t + \sigma(t,X_t)\rd B_t,
\end{equation}
where 
\begin{equation*}
    u(t,x) = \int_{\R^d} \E[K(x-X_t)|X_0=y]\cdot w(y)\rd y.
\end{equation*}
Here, $B$ is a $d$-dimensional Brownian motion and $\sigma^T\sigma$ is uniformly elliptic and bounded. $K$ is a $d\times d$-matrix-valued Borel measurable function such that $|K(x)|\lesssim |x|^{-\gamma}$ with $\gamma\in[0,d)$. Moreover, $w:\R^d\to \R^d$ is some given function that is bounded and integrable. The drift term in this stochastic system depends on the law of its solution, and such an equation was first studied by Mckean \cite{Mckean-1966} and known as the Mckean-Vlasov stochastic differential equations.  

This system of stochastic differential equations arises from the random vortex method in fluid dynamics. When $\sigma$ is a positive constant and $d=2$, it is well-known that the dynamics of \eqref{stochastic-vortex-SDE} is associated with the 2D vorticity equations for incompressible fluid flows with viscosity $\sigma>0$
\begin{equation}
\label{vorticity-equation}
    \frac{\partial W}{\partial t} + u\cdot \nabla W = \sigma\Delta w, 
\end{equation}
where $W=\nabla \wedge u$ is the vorticity of the flow, whose initial data $W(0,x) = w(x)$ is given. The velocity can be given by the Biot-Savart law
\[
u(t,x) = \int_{\R^2} K(x-y)\cdot W(t,y)\rd y
\]
solving the incompressible Navier-Stokes equation, and $K$ is the singular kernel 
\begin{equation}
\label{Biot-Savart-kernel}
K(x) = \left(\frac{\partial G}{\partial x_2},-\frac{\partial G}{\partial x_1}\right),
\end{equation}
where $G(x) = -\frac{\ln |x|}{2\pi}$ is the fundamental solution of the Poisson equation. 
In \cite{Chorin-1973}, Chorin introduced the random vortex method by splitting the 2D Navier-Stokes equation into an Euler's equation and a heat equation, where the latter can be simulated using random walks. The rate of convergence of the method was studied in Beale and Majda \cite{Beale-Majda-1981}, further improved in Goodman \cite{Goodman-1987} and Long \cite{Long-1988}. 

Meanwhile, the vorticity $W$ that solves equation \eqref{vorticity-equation} can be obtained as the mean-field limit via propagation of chaos for the interacting particle system of $N$-vortices
\begin{equation}
\label{N-vortices system}
    \rd X^{N,i}_t = \frac{1}{N} \sum_{j\neq i} w^{N,j} K(X^{N,i}_t-X^{N,j}_t)\rd t +\sigma \rd B^i_t,\quad \forall i=1,2,\cdots,N,
\end{equation}
where $w^{N,i}\in \R$ represents the intensity of the vortices, and $B^i$'s, $i=1,2,\cdots,N$, are independent two-dimensional standard Brownian motions. In \cite{Marchioro-Pulvirenti-1982}, Marchioro and Pulvirenti (see also \cite{Meleard-2000}) approximated the vorticity equation by the regularised $N$-vortices system and proved the propagation of chaos for incompressible viscous two-dimensional fluids with bounded integrable initial vorticity. Without regularising the interaction kernel, Osada \cite{Osada-1985} used the transition probability density associated with the generators of generalised divergence form and proved that the system \eqref{N-vortices system} defines a conservative diffusion, which then implied the well-posedness of \eqref{N-vortices system}. He also proved the propagation of chaos for the two-dimensional incompressible Navier-stokes equation when the viscosity is large in \cite{Osada-1986} and when the viscosity is small in \cite{Osada-1987}. More recently, Fournier, Hauray and Mischler proved a stronger propagation of chaos of trajectories in \cite{Fournier-Hauray-Mischler-2014}. Indeed, the propagation of chaos for interacting particle systems has received a lot of attention ever since it was first introduced by Mckean \cite{Mckean-1966,Mckean-1967}. See for example \cite{Sznitman-1991,Jabin-2014,Jabin-Wang-2017} for detailed reviews. 

In \cite{Qian-Yao-2022}, the authors considered the limiting equation of \eqref{N-vortices system}
was considered and proved the existence and uniqueness of weak and strong solutions of the equation. In this paper, instead of additive noise, we will study such Mckean-Vlasov type stochastic differential equations but with multiplicative noise, which can be viewed as the mean-field limit of the stochastic vortex system 
\begin{equation*}
     \rd X^{N,i}_t = \frac{1}{N} \sum_{j\neq i} w^{N,j} K(X^{N,i}_t-X^{N,j}_t)\rd t +\sigma(t,X^{N-i}_t) \rd B^i_t,\quad \forall i=1,2,\cdots,N,
\end{equation*}
which are more realistic vortex models as the noise now depends on the relative positions of the vortices. It is worth mentioning that Euler equations with multiplicative noise in the Stratonovich form were also studied in the work \cite{Flandoli-Gubinelli-Priola-2011} by Flandoli, Gubinelli and Priola, where they handled the stochastic vortex system with a finite number of vortices. The difficulty lies in the singularity of the interacting kernel, and to deal with it, we will establish a new estimate for transition probability densities of diffusion in Section 3. Then we will use this estimate and fixed point theorem to conclude the existence and uniqueness results as desired in Section 4. Before that, let us introduce some handy notations and a couple of useful known results in the next Section.

\section{Preliminaries}
In this section, we will introduce our notations and assumptions on the equation coefficients. Then we will present a couple of useful results for the proofs in sections 3 and 4. 

Let $X^b = \{X^b_t\}_{t\geq 0}$ be a diffusion process in $\R^d$ that satisfies the following stochastic differential equation with a measurable and bounded drift $b$:
\begin{equation}
\label{SDE-b-drift}
    \rd X^{b}_t = b(t,X^{b}_t)\rd t + \sigma(t,X^{b}_t)\rd B_t,  
\end{equation}
where $\{B_t\}_{t\geq 0}$ is a $d$-dimensional standard Brownian motion on some probability space $(\Omega,\mathscr{F},\mathbb{P})$. Let us denote the process solves the equation with zero drift by$X^0=\{X^0_t\}_{t\geq 0}$ 
\begin{equation}
\label{SDE-no-drift}
    \rd X^0_t = \sigma(t,X^0_t)\rd B_t.
\end{equation}
We shall use $p_b(s,x,t,y)$ to denote the transition probability density function of the process $\{X^b\}$, i.e. for any Borel measurable set $A\subset \R^d$, $\forall x\in \R^d$, $0\leq s<t$,
\begin{equation*}
    \mathbb{P}(X^b_t\in A|X^b_s=x)= \int_A p_b(s,x,t,y)\rd y.
\end{equation*}
These transition probability densities are known to be continuous in space and time. When $s=0$, we shall simplify our notation and write $p_b(x,t,y)$ for the transition probability density. Let us denote the transition probability of $\{X^0_t\}$ by $p$, i.e. $p(s,x,t,y)=p_0(s,x,t,y)$. Lastly, we use $\mathbb{P}^x$ to denote the conditional probability such that 
\[
\mathbb{P}^x(X^b_t\in A) = \mathbb{P}(X^b_t\in A|X^b_0=x) 
\]
for all Borel measurable $A\subset\R^d$.

Set $g(t,x) = \sigma^T(t,x)\sigma(t,x)$. Throughout the paper, we assume that there exists some constant $\xi>0$ such that for all $i,j=1,2,\cdots,d$, 
\begin{equation*}
    \frac{1}{\xi}\leq |g^{ij}(t,x)|\leq \xi,\qquad \forall t\geq 0 , \forall x\in \R^d,
\end{equation*} 
and $g$ has bounded derivatives. We note here that under the assumption, the symmetric matrix $(g^{ij})_{1\leq i,j\leq d}$ is positive definite. We use the lower index to denote the inverse, i.e. $(g_{ij})_{1\leq i,j\leq d}$ represents the inverse matrix of $(g^{ij})_{1\leq i,j\leq d}$, and we use $\langle \cdot, \cdot \rangle_g$ to denote the inner product with respect to $(g_{ij})$, i.e. for two $d$-dimensional vectors $a$ and $b$, 
\[
\langle a,b\rangle_g = \sum_{1\leq i,j\leq d} g_{ij} a^i b^j,
\]
and $|a|_g^2 = \langle a,a\rangle_g$. The gradient operator $\nabla^g_x$ is defined by 
\[
(\nabla^g_x f(x))^i = \sum_j g^{ij}(x) \frac{\partial f}{\partial x_j},\quad i=1,2,\cdots,d
\] 
for all $f\in C^1(\R^d)$.

Let us consider the integral kernel $K$ on $\R^d$ such that
\[
|K(x)|\leq \frac{\alpha}{|x|^\gamma}, \quad \forall x\neq 0\in \R^d,
\]
where $\alpha >0$ is some constant and $\gamma\in[0,d)$. Let $\mathcal{B}$ denote the complete metric space of all bounded and measurable functions on $\R_+\times \R^d$, equipped with $\lVert \cdot \rVert_\infty$ and for each $y\in\R^d$. We use $\mathcal{B}([0,T]\times \R^d)$ to denote the space of all bounded and measurable functions on $[0,T]\times \R^d$, and for each constant $L>0$, we set 
\[
\mathcal{B}_L([0,T]\times\R^d) = \{b\in \mathcal{B}([0,T]\times\R^d): \lVert b\rVert_\infty \leq L\}.
\]
For each $b\in \mathcal{B}$, we define an operator $\K$ on $\mathcal{B}$ by 
\begin{equation*}
    \K(b)(t,x) = \int_{\R^d} \E^y\left[K(x-X^b_t)\right]\cdot w(y)\rd y = \int_{\R^d} \E[K(x-X^b_t)|X^b_0=y]\cdot w(y)\rd y,
\end{equation*}
where $\E^y$ represents the expectation under the conditional measure $\mathbb{P}^y$, $w\in L^1\cap L^\infty(\R^d)$ and $\{X^b_t\}_{t\geq 0}$ is the diffusion process that satisfies \eqref{SDE-b-drift}.

Our goal is to show that the Mckean-Vlasov equation 
\begin{equation}
    \label{Mckean-Vlasov-equation}
    \rd X_t = u(t,X_t) \rd t +  \sigma(t,X_t)\rd B_t,
\end{equation}
where 
\[
u(t,x)= \int_{\R^d} \E^y\left[K(x-X_t)\right]\cdot w(y)\rd y,
\]
has a unique weak solution up to some fixed time. To this end, we will need the following results. 

\begin{proposition}[Theorem 2.4, \cite{Qian-Zheng-2004}]
\label{Qian-Zheng-proposition}
Under the above notations, we have that for every $b\in \mathcal{B}$, the transition probability density functions $p_b$ and $p$ satisfy 
\begin{equation}
    \label{trans-prob-formula}
    p_b(x,t,y) = p(x,t,y) + \int^t _0 \E^x\left[U^b_s\langle b(s,X^0_s),\nabla_x^g p(X^0_s,t-s,y)\rangle_g\right]\rd s,
\end{equation}
where
\begin{equation}
\label{radon-nikodym-deriv}
    U^b_t = \exp\left(\int^t_0 \langle b(s,X^0_s), \rd M_s\rangle_g -\frac{1}{2}\int^t_0 \left\vert b(s,X^0_s) \right\vert^2_g \rd s\right),\quad \forall t\geq 0,
\end{equation}
and $M$ is the martingale part of $X^0$ such that 
\begin{equation*}
    \langle M^i, M^j\rangle_t = \int^t_0 g^{ij}(s,X^0_s)\rd s.
\end{equation*}

\end{proposition}

The second result we need is an estimate of the derivatives of the transition probability densities. We shall state the result in the form that applies to our case.

\begin{proposition}[Theorem 3.3.11, \cite{Stroock-2008}]
\label{Stroock-result}
  Let $p$ be the transition probability density function associated with the diffusion \eqref{SDE-no-drift}. Then there exists some constant $A$, depending on $\xi$ and second order derivative of $g$ such that 
\begin{equation*}
    \left\vert \frac{\partial p}{\partial x_j}(x,t,y)\right\vert \leq \frac{A}{1\wedge t^\frac{1+d}{2}}\exp\left(-\left(At-\frac{|y-x|^2}{At}\right)^-\right)
\end{equation*}
for all $t>0$, $x,y\in\R^d$ and $j=1,2,\cdots,d$.
\end{proposition}

Finally, we will also need the following classical Aronson's estimate on the transition probability $p$ for the diffusion \eqref{SDE-no-drift}:

\begin{proposition}[\cite{Aronson-1967}]
\label{Aronson-result}
    There exist two positive constants $\kappa'$ and $\kappa$, depending only on the bounds of $g$, such that 
    \begin{equation*}
        \frac{\kappa'}{t^{\frac{d}{2}}}e^{-\frac{|y-x|^2}{\kappa' t}}\leq p(x,t,y) \leq  \frac{\kappa}{t^{\frac{d}{2}}}e^{-\frac{|y-x|^2}{\kappa t}}.
    \end{equation*}
\end{proposition}

\section{Transition probability densities}

In this section, we will establish a useful bound of the transition probability density $p_b$. This bound is sharper than the Aronson bound in Proposition \ref{Aronson-result}.

\begin{theorem}
     Let $1<q<\frac{d}{d-1}$ and 
     \begin{equation}
     \label{constant-C}
         C = \max \left\{\frac{2d^2 q}{d-dq+q}\xi^2 A e^A  \kappa ^\frac{1}{q} \left(\frac{A \kappa  \pi}{A\wedge  \kappa q}\right)^\frac{d}{2q}, A\vee  \kappa q \right\}, 
     \end{equation}
     where $A$ and $\kappa$ are the constants in Propositions \ref{Stroock-result} and \ref{Aronson-result}. 
     Then it holds 
     \begin{equation}
     \label{trans-prob-estimate}
         p_b(x,t,y)\leq p(x,t,y) + \lVert b\rVert_\infty \sqrt{t} e^{\frac{\xi}{2(q-1)}\lVert b\rVert_\infty^2 t} \frac{C}{t^\frac{d}{2}}e^{-\frac{|y-x|^2}{Ct}}
     \end{equation}
     for all $t\in[0,1]$, $x,y\in\R^d$.
\end{theorem}

\begin{proof}
It follows from \eqref{trans-prob-formula} in Proposition \ref{Qian-Zheng-proposition} that for any $t\geq 0$, $x,y\in\R^d$,
\begin{equation*}
    I := \left\vert p_b(x,t,y)-p(x,t,y)\right\vert\leq \xi \int^t_0 \E^x\left[U^b_s|b(s,X^0_s)||\nabla_x^g p(X^0_s,t-s,y) | \right]\rd s. 
\end{equation*}
Moreover, by Cauchy-Schwarz inequality,
\begin{equation}
\label{I-bound}
\begin{aligned}
    \left|\nabla_x^g p(X^0_s,t-s,y) \right| = & \sqrt{\sum_i\left(\sum_j g^{ij}\frac{\partial p}{\partial x_j}\right)^2}\\
                                 \leq & d\xi \sqrt{\sum_j \left(\frac{\partial p}{\partial x_j}\right)^2} =d\xi \lvert \nabla_x p\rvert. 
\end{aligned}
\end{equation}
Therefore, using the boundedness of $b$ and H{\" o}lder's inequality, we have that for $p,q>1$ with $\frac{1}{p}+\frac{1}{q}=1$,
\begin{equation*}
    \begin{aligned}
        I\leq & d\xi^2\int^t_0 \E^x\left[U^b_s|b(s,X^0_s)|\left|\nabla_x p(X^0_s,t-s,y)\right|\right]\rd s\\
        \leq & d\xi^2 \lVert b\rVert_\infty \int^t_0 \E^x [|U^b_s|^p]^\frac{1}{p}\E^x\left[\left|\nabla_x p(X^0_s,t-s,y)\right|^q\right]^{\frac{1}{q}}\rd s.
    \end{aligned}
\end{equation*}
Now an application of the estimate in Proposition \ref{Stroock-result} yields that   
\begin{equation*}
\begin{aligned}   
\left\vert \frac{\partial p}{\partial x_j}(X^0_s,t-s,y)\right\vert &\leq \frac{A}{1\wedge(t-s)^\frac{1+d}{2}}\exp\left(-\left(A(t-s)-\frac{|y-X^0_s|^2}{A(t-s)}\right)^-\right)\\
&\leq \frac{A}{1\wedge(t-s)^\frac{1+d}{2}}\exp\left(A(t-s)-\frac{|y-X^0_s|^2}{A(t-s)}\right).
\end{aligned}
\end{equation*}
Therefore, for $0\leq s\leq t\leq 1$, we have
\begin{equation*}
    \begin{aligned}
        \left|\nabla_x p(X^0_s,t-s,y)\right|\leq  e^{A(t-s)} \frac{dA}{(t-s)^\frac{1+d}{2}}\exp\left(-\frac{|y-X^0_s|^2}{A(t-s)}\right).
    \end{aligned}
\end{equation*}
Consequently, we have 
\begin{equation*}
    \begin{aligned}
        \E^x\left[\left|\nabla_x p(X^0_s,t-s,y)\right|^q\right]\leq & \frac{(dA)^qe^{qA(t-s)}}{(t-s)^{\frac{q(1+d)}{2}}}
        \E^x\left[\exp\left(-\frac{q|y-X^0_s|^2}{A(t-s)}\right)\right].
    \end{aligned}
\end{equation*}
By the Aronson's bound \eqref{Aronson-bound} in Proposition \ref{Aronson-result}, there exists some $\kappa>0$ such that  
\begin{equation}
\label{Aronson-bound}
    p(x,t,y)\leq \frac{\kappa}{t^{\frac{d}{2}}}e^{-\frac{|y-x|^2}{\kappa t}},
\end{equation}
so 
\begin{equation*}
    \begin{aligned}
        \E^x\left[\exp\left(-\frac{q|y-X^0_s|^2}{A(t-s)}\right)\right] =& \int_{\R^d} e^{-\frac{q|y-z|^2}{A(t-s)}}p(z,s,x)\rd z\\
        \leq & \int_{\R^d} e^{-\frac{q|y-z|^2}{A(t-s)}}\frac{ \kappa }{s^{\frac{d}{2}}}e^{-\frac{|z-x|^2}{ \kappa s}}\rd z\\
        =& \frac{ \kappa }{s^{\frac{d}{2}}} \int_{\R^d}e^{-\frac{q|y-x-u|^2}{A(t-s)}}e^{-\frac{|u|^2}{ \kappa s}}\rd u\\
        =& \frac{ \kappa }{s^{\frac{d}{2}}} \left(2\pi\frac{ A(t-s)}{2q}\right)^\frac{d}{2} \left(2\pi\frac{  \kappa s}{2}\right)^\frac{d}{2} \frac{1}{\left(2\pi\left(\frac{ A(t-s)}{2q}+\frac{  \kappa s}{2}\right)\right)^\frac{d}{2}}e^{-\frac{|y-x|^2}{\frac{ A(t-s)}{q}+ \kappa s}}\\
        =& (A  \kappa \pi)^\frac{d}{2}\frac{ \kappa (t-s)^\frac{d}{2}}{\left(A(t-s)+ \kappa sq\right)^\frac{d}{2}}e^{-\frac{q|y-x|^2}{ A(t-s)+ \kappa sq}}.
    \end{aligned}
\end{equation*}
Therefore, as $(A\wedge  \kappa q)t\leq A(t-s)+ \kappa sq\leq (A\vee  \kappa q)t$, we deduce that 
\begin{equation}
\label{grad-p-Lq-bound}
    \begin{aligned}
        \E^x\left[\left|\nabla_x p(X^0_s,t-s,y)\right|^q\right]^{\frac{1}{q}}\leq & \frac{(dA)e^{A(t-s)}}{(t-s)^{\frac{(1+d)}{2}}}
        (A \kappa \pi)^\frac{d}{2q} \frac{ \kappa ^\frac{1}{q}(t-s)^\frac{d}{2q}}{((A\wedge  \kappa q)t)^\frac{d}{2q}}e^{-\frac{|y-x|^2}{(A\vee  \kappa q)t}}\\
        \leq & dA e^A  \kappa ^\frac{1}{q} \left(\frac{A \kappa  \pi}{A\wedge  \kappa q}\right)^\frac{d}{2q} (t-s)^{\frac{d}{2q}-\frac{d}{2}-\frac{1}{2}} \frac{1}{t^\frac{d}{2q}} e^{-\frac{|y-x|^2}{(A\vee  \kappa q)t}}.
    \end{aligned}
\end{equation}
Meanwhile, we also have
\begin{equation}
\label{U-Lp-bound}
    \begin{aligned}
       &\E^x[|U^b_s|^p]\\ =& \E^x\left[\exp\left(p\int^s_0 \langle b(u,X_u),\rd M_u\rangle_g -\frac{ p^2}{2}\int^s_0|b(u,X_u)|^2_g\rd u\right)\right.\\
       &\left.\cdot \exp\left(\frac{p(p-1)}{2}\int^s_0|b(u,X_u)|^2_g\rd u\right)  \right]\\
                        \leq & e^{\frac{\xi p(p-1)}{2}\lVert b\rVert_\infty^2 s}\E^x\left[\exp\left(p\int^s_0 \langle b(u,X_u),\rd M_u\rangle_g -\frac{ p^2}{2}\int^s_0|b(u,X_u)|^2_g\rd u\right)\right]\\
                        \leq & e^{\frac{\xi p(p-1)}{2}\lVert b\rVert_\infty^2 s},
    \end{aligned}
\end{equation}
where we utilise the fact that the exponential is a (super)martingale. Consequently, combining \eqref{I-bound}, \eqref{grad-p-Lq-bound} and \eqref{U-Lp-bound} we conclude that 
\begin{equation*}
    \begin{aligned}
                I \leq & d^2 \xi^2 A e^A  \kappa ^\frac{1}{q} \left(\frac{A \kappa  \pi}{A\wedge  \kappa q}\right)^\frac{d}{2q} \lVert b\rVert_\infty \frac{1}{t^\frac{d}{2q}} e^{-\frac{|y-x|^2}{(A\vee  \kappa q)t}} \int^t_0 e^{\frac{\xi (p-1)}{2}\lVert b\rVert_\infty^2 s}  (t-s)^{\frac{d}{2q}-\frac{d}{2}-\frac{1}{2}} \rd s \\  
                \leq & d^2 \xi^2 A e^A  \kappa ^\frac{1}{q} \left(\frac{A \kappa  \pi}{A\wedge  \kappa q}\right)^\frac{d}{2q} \lVert b\rVert_\infty \frac{1}{t^\frac{d}{2q}} e^{\frac{\xi (p-1)}{2}\lVert b\rVert_\infty^2 t-\frac{|y-x|^2}{(A\vee  \kappa q)t}} \int^t_0 (t-s)^{\frac{d}{2q}-\frac{d}{2}-\frac{1}{2}} \rd s. 
    \end{aligned}
\end{equation*}
When we choose $q$ such that 
\[
\frac{d}{2q}-\frac{d}{2}-\frac{1}{2}>-1,
\]
i.e. $q<\frac{d}{d-1}$ and $p>d$, then the integral converges, and implies that 
\begin{equation*}
    \begin{aligned}
        I \leq & \frac{2d^2 q}{d-dq+q}\xi^2 A e^A  \kappa ^\frac{1}{q} \left(\frac{A \kappa  \pi}{A\wedge  \kappa q}\right)^\frac{d}{2q} \lVert b\rVert_\infty \sqrt{t}e^{\frac{\xi (p-1)}{2}\lVert b\rVert_\infty^2 t}\frac{1}{t^\frac{d}{2}} e^{-\frac{|y-x|^2}{(A\vee  \kappa q)t}},
    \end{aligned}
\end{equation*}
which yields the desired result.
\end{proof}

 \begin{remark}
For any $0\leq \tau < t\leq 1$, it also holds
 \begin{equation*}
         p_b(\tau,x,t,y)\leq p(\tau,x,t,y) + \lVert b\rVert_\infty \sqrt{t-\tau} e^{\frac{\xi}{2(q-1)}\lVert b\rVert_\infty^2 (t-\tau)} \frac{C}{(t-\tau)^\frac{d}{2}}e^{-\frac{|y-x|^2}{C(t-\tau)}}
     \end{equation*}
     for all $x,y\in\R^d$, where $C$ is the same as in \eqref{constant-C}.
 \end{remark}
\section{Main Results}

To facilitate the proof of the main result, let us prove two lemmas. For convenience, let us introduce some new notations. Let $R>0$. Set 
\begin{equation*}
    K_{B_R}(x) = K(x)1_{B_R}(x),\qquad K_{B_R^C}(x) = K(x)1_{B_R^C}(x),
\end{equation*}
for all $x\neq 0 \in\R^d$, where $B_R$ is the ball centred at the origin with radius $R$, and $B_R^C$ its complement. 
\begin{lemma}
\label{technical-lemma}
    Let $R>0$ be a constant. Then 
    \begin{equation}
    \label{B_R-bound}
        \int_{\R^d} \E^y\left[\left|K_{B_R}(x-X^b_t)\right|\right]|w(y)|\rd y\leq \frac{2\alpha C^{1+\frac{d}{2}}\pi^d R^{d-\gamma} \lVert w\rVert_\infty }{\Gamma(\frac{d}{2})(d-\gamma)} \left(1+\lVert b\rVert_\infty \sqrt{t} e^{\frac{\xi}{2(q-1)}\lVert b\rVert_\infty^2 t} \right),
    \end{equation}
    \begin{equation}
    \label{B_R-complement-bound}
         \int_{\R^d} \E^y\left[\left|K_{B_R^C}(x-X^b_t)\right|\right]|w(y)|\rd y \leq \frac{\alpha \lVert w\rVert_1}{R^\gamma},
    \end{equation}
    for all $x\in \R^d$ and $t\in [0,1]$.
\end{lemma}

\begin{proof}
    Using the transition probability density function $p_b$, we can write 
    \begin{equation*}
        \begin{aligned}
            \int_{\R^d} \E^y\left[\left|K_{B_R}(x-X^b_t)\right|\right]|w(y)|\rd y \leq & \int_{\R^d}\int_{\R^d} |K_{B_R}(x-z)|p_b(y,t,x)||w(y)|\rd y\rd z\\
            \leq &\int_{\R^d}\left(\int_{B_R} \frac{\alpha}{|u|^\gamma}p_b(y,t,x-u)|w(y)|\rd u\right)\rd y,
        \end{aligned}
    \end{equation*}
    which, after applying \eqref{trans-prob-estimate} and the Aronson's bound \eqref{Aronson-bound}, implies that 
    \begin{equation*}
        \begin{aligned}
            &\int_{\R^d} \E^y\left[\left|K_{B_R}(x-X^b_t)\right|\right]|w(y)|\rd y\\ 
            \leq & \int_{\R^d}\left(\int_{B_R} \frac{\alpha}{|u|^\gamma} \left(p(y,t,x-u) + \lVert b\rVert_\infty \sqrt{t} e^{\frac{\xi}{2(q-1)}\lVert b\rVert_\infty^2 t} \frac{C}{t^\frac{d}{2}}e^{-\frac{|x-y-u|^2}{Ct}}\right)\rd u\right)|w(y)|\rd y\\
            \leq & \int_{\R^d}\left(\int_{B_R}\frac{\alpha}{|u|^\gamma}\left(\frac{ \kappa }{t^{\frac{d}{2}}}e^{-\frac{|x-y-u|^2}{ \kappa t}}+ \lVert b\rVert_\infty \sqrt{t} e^{\frac{\xi}{2(q-1)}\lVert b\rVert_\infty^2 t} \frac{C}{t^\frac{d}{2}}e^{-\frac{|x-y-u|^2}{Ct}}\right)\rd u\right)|w(y)|\rd y \\
            \leq & \lVert w\rVert_\infty \left(1+\lVert b\rVert_\infty \sqrt{t} e^{\frac{\xi}{2(q-1)}\lVert b\rVert_\infty^2 t} \right) \int_{\R^d} \int_{B_R} \frac{\alpha}{|u|^\gamma} \frac{C}{t^\frac{d}{2}}e^{-\frac{|x-y-u|^2}{Ct}} \rd u\rd y,
        \end{aligned}
    \end{equation*}
    where we have used that the constant $C$ given in \eqref{constant-C} is greater than $ \kappa $. Since 
    \begin{equation*}
        \int_{\R^d} \frac{C}{t^\frac{d}{2}}e^{-\frac{|x-y-u|^2}{Ct}} \rd y =\pi^\frac{d}{2} C^{1+\frac{d}{2}}
    \end{equation*}
    and 
    \begin{equation*}
        \int_{B_R} \frac{\alpha}{|u|^\gamma}\rd u = \alpha \int^R_0 \frac{1}{r^\gamma} r^{d-1}\rd r S_{d-1},
    \end{equation*}
    where 
    \begin{equation*}
        S_{d-1} = \int^{2\pi}_0 \int^\pi_0 \cdots \int^\pi_0 \sin^{d-2}(\phi_1)\cdots \sin^2(\phi_{d-3})\sin(\phi_{d-2})\rd \phi_1\cdots \rd \phi_{d-2}\rd \phi_{d-1} = \frac{2\pi^\frac{d}{2}}{\Gamma(\frac{d}{2})}
    \end{equation*}
    is the surface area of $(d-1)$-sphere, we conclude \eqref{B_R-bound}. As for \eqref{B_R-complement-bound}, the proof is straightforward. For all $x\in \R^d$ and $t\geq 0$, 
    \begin{equation*}
        \begin{aligned}
            \int_{\R^d} \E^y\left[\left|K_{B_R^C}(x-X^b_t)\right|\right]|w(y)|\rd y \leq & \int_{\R^d}\int_{B_R^C}|K(u)|p_b(y,t,x-u)|w(y)|\rd u\rd y\\
            \leq & \int_{\R^d}\int_{B_R^C} \frac{\alpha}{|u|^\gamma} p_b(y,t,x-u)|w(y)|\rd u\rd y\\
            \leq & \frac{\alpha}{R^\gamma} \int_{\R^d} \left(\int_{B_R^C} p_b(y,t,x-u)\rd u\right)|w(y)|\rd y\\
            \leq & \frac{\alpha\lVert w\rVert_1}{R^\gamma}.
        \end{aligned}
    \end{equation*}
    Therefore, the proof is complete.
\end{proof}

Now let us set 
\begin{equation*}
    C_0 = \frac{\alpha\lVert w\rVert_1}{R^\gamma} + \frac{2\alpha C^{1+\frac{d}{2}}\pi^d R^{d-\gamma}\lVert w\rVert_\infty }{\Gamma(\frac{d}{2})(d-\gamma)} \left(1+ e^{\frac{\xi}{2(q-1)}} \right). 
\end{equation*}
For each $L\geq C_0\vee 1$, let $T_L:=\frac{1}{L^2}\leq 1$, and 
\[
\mathcal{B}_L([0,T_L]\times\R^d) = \{b\in \mathcal{B}([0,T_L]\times\R^d): \lVert b\rVert_\infty \leq L\}.
\]
Then we have the following result which tells that $\K(b)$ is a mapping from $\mathcal{B}_L([0,T_L]\times\R^d)$ to itself.
\begin{lemma}
    For each $b\in\mathcal{B}_L([0,T_L]\times\R^d)$, $\K(b)\in \mathcal{B}_L([0,T_L]\times\R^d)$.
\end{lemma}
\begin{proof}
    We only need to show that $\K(b)$ is bounded above by $L$ for every $b\in \mathcal{B}_L([0,T_L]\times\R^d) $. Let $R>0$. We notice that for every $t\in [0,T_L]$ and $x\in\R^d$,
    \begin{equation*}
        \begin{aligned}
            |\K(b)(t,x)| \leq& \int_{\R^d} \E^y\left[\left|K_{B_R}(x-X^b_t)\right|\right]|w(y)|\rd y +\int_{\R^d} \E^y\left[\left|K_{B_R^C}(x-X^b_t)\right|\right]|w(y)|\rd y.
        \end{aligned}
    \end{equation*}
    Now let us apply the estimates \eqref{B_R-bound} and \eqref{B_R-complement-bound}. Then we deduce that 
    \begin{equation*}
        \begin{aligned}
            |\K(b)(t,x)| \leq& \frac{2\alpha C^{1+\frac{d}{2}}\pi^d R^{d-\gamma} \lVert w\rVert_\infty }{\Gamma(\frac{d}{2})(d-\gamma)} \left(1+\lVert b\rVert_\infty \sqrt{t} e^{\frac{\xi}{2(q-1)}\lVert b\rVert_\infty^2 t} \right)+\frac{\alpha \lVert w\rVert_1}{R^\gamma}.
        \end{aligned}
    \end{equation*}
    Since $\lVert b\rVert_\infty \leq L$ and $t\leq T_L$, $\lVert b\rVert_\infty \sqrt{t}\leq 1$, and thus 
    \begin{equation*}
        |\K(b)(t,x)| \leq \frac{2\alpha C^{1+\frac{d}{2}}\pi^d R^{d-\gamma} \lVert w\rVert_\infty }{\Gamma(\frac{d}{2})(d-\gamma)} \left(1+ e^{\frac{\xi}{2(q-1)}} \right)+\frac{\alpha \lVert w\rVert_1}{R^\gamma}\leq L.
    \end{equation*}
Therefore, we conclude that $\K(b)\in \mathcal{B}_L([0,T_L]\times \R^d)$.
\end{proof}

Now we are ready to prove the main result, which is to show the existence and uniqueness of the weak solution by showing that $\K$ is a contraction. 

\begin{theorem}
Let $0<\tau <(1\wedge \frac{1}{\xi +\sqrt{\xi}})T_L$. 
 Then $\K:\mathcal{B}_L([0,\tau]\times \R^d) \to \mathcal{B}_L([0,\tau]\times \R^d)$ is a contraction. Moreover, this implies that the equation \eqref{Mckean-Vlasov-equation} has a unique weak solution up to time $\tau$.
\end{theorem}
\begin{proof}
Let $b$ and $\tilde{b}$ be two bounded and measurable functions in $\mathcal{B}_L([0,T_L]\times \R^d)$. Then for every $t\in[0,T_L]$ and $x\in\R^d$:
\begin{equation*}
    \begin{aligned}
        \left|\K(b)(t,x)-\K(\tb)(t,x)\right|=\left|\int_{\R^d}\E^y\left[K(x-X^0_t)(U^b_t-U^{\tb}_t)\right]\cdot w(y)\rd y\right|,
    \end{aligned}
\end{equation*}
where $X^0_t$ satisfies \eqref{SDE-no-drift}, and $U^b$ and $U^{\tb}$ are defined as in \eqref{radon-nikodym-deriv}, i.e. 
\begin{equation*}
    U^b_t = \exp\left(\int^t_0 \langle b(s,X^0_s), \rd M_s\rangle_g -\frac{1}{2}\int^t_0 \left\vert b(s,X^0_s) \right\vert^2_g \rd s\right).
\end{equation*}
For simplicity, let us denote the exponent of the Radon-Nikodym derivative by $N_b$, i.e. 
\begin{equation*}
    N^b_t=\int^t_0 \langle b(s,X^0_s), \rd M_s\rangle_g -\frac{1}{2}\int^t_0 \left\vert b(s,X^0_s) \right\vert^2_g \rd s. 
\end{equation*}
Then by the mean value theorem, there exists some $\theta\in(0,1)$ such that 
\begin{equation*}
\begin{aligned}
    U^b_t-U^{\tb}_t = & e^{\theta N^b_t+(1-\theta)N^{\tb}_t}(N^b_t-N^{\tb}_t),
\end{aligned}
\end{equation*}
where $b_\theta=\theta b+(1-\theta) \tb$. 

Notice that 
\begin{equation}
\label{N-intermediate}
    \begin{aligned}
        &\theta N^b_t+(1-\theta)N^{\tb}_t\\ 
        =& \int^t_0 \langle (\theta b+(1-\theta\tb))(s,X^0_s),\rd M_s\rangle_g -\frac{1}{2}\int^t_0 \theta|b(s,X^0_s)|_g^2+(1-\theta)|\tb(s,X^0_s)|_g^2 \rd s\\
        =& N^{b_\theta}_t +\frac{1}{2}\theta(\theta -1) |b(s,X^0_s)|_g^2 +2\theta(1-\theta)\langle b(s,X^0_s),\tb(s,X^0_s)\rangle_g -\theta(1-\theta)|\tb(s,X^0_s)|_g^2\rd s\\
        =& N^{b_\theta}_t-\frac{1}{2}\theta(1-\theta)\int^t_0 |b(s,X^0_s)-\tb(s,X^0_s)|_g^2\rd s,
    \end{aligned}
\end{equation}
where 
\begin{equation*}
    N^{b_\theta}_t = \int^t_0 \langle \theta b(s,X^0_s) +(1-\theta)\tb(s,X^0_s),\rd M_s\rangle_g-\frac{1}{2}\int^t_0 |\theta b(s,X^0_s)+(1-\theta)\tb(s,X^0_s)|_g^2\rd s.
\end{equation*}
Meanwhile, we can decompose the difference into two parts as
\begin{equation*}
        N^b_t-N^{\tb}_t = Z_t + A_t,
\end{equation*}
where 
\begin{equation}
\label{Z-martingale}
    Z_t = \int^t_0 \langle b(s,X^0_s)-\tb(s,X^0_s),\rd M_s\rangle_g
\end{equation}
is the martingale part and 
\begin{equation*}
    A_t =-\frac{1}{2}\int^t_0 (|b(s,X^0_s)|_g^2 - |\tb(s,X^0_s)|_g^2)\rd s.
\end{equation*}
Consequently, since $b,\tb\in \mathcal{B}_L([0,T_L]\times\R^d)$, we may conclude that 
\begin{equation}
\label{N-difference}
    \begin{aligned}
        |N^b_t-N^{\tb}_t|\leq & |Z_t| +|A_t|\\
        \leq & |Z_t| +\frac{1}{2}\xi \int^t_0 \left(|b(s,X^0_s)|+|\tb(s,X^0_s)|\right)\lVert b-\tb\rVert_\infty \rd s\\
        \leq & |Z_t| + \xi L \lVert b-\tb\rVert_\infty t.
    \end{aligned}
\end{equation}
Using \eqref{N-intermediate} and \eqref{N-difference}, we get that 
\begin{equation*}
    \begin{aligned}
        |U^b_t-U^{\tb}_t| =& 
        U^{b_\theta}_t e^{-\frac{1}{2}\theta(1-\theta)\int^t_0 |b(s,X^0_s)-\tb(s,X^0_s)|_g^2\rd s}|N^b_t-N^{\tb}_t|\\
        \leq & U^{b_\theta}_t(|Z_t| + \xi L \lVert b-\tb\rVert_\infty t).
    \end{aligned}
\end{equation*}

It follows from the above estimate immediately that 
\begin{equation*}
    \begin{aligned}
         \left|\K(b)(t,x)-\K(\tb)(t,x)\right|\leq & \int_{\R^d} \E^y\left[|K(x-X^0_t)||U^b_t-U^{\tb}_t|\right]|w(y)|\rd y\\
         \leq & \int_{\R^d} \E^y\left[|K_{B_R}(x-X^0_t)||Z_t|U^{b_\theta}_t\right]|w(y)|\rd y\\
         &+ \int_{\R^d} \E^y\left[|K_{B_R^C}(x-X^0_t)||Z_t|U^{b_\theta}_t\right]|w(y)|\rd y\\
         &+ \xi L \lVert b-\tb\rVert_\infty t \int_{\R^d} \E^y\left[|K_{B_R}(x-X^0_t)|U^{b_\theta}_t\right]|w(y)|\rd y\\
         &+ \xi L \lVert b-\tb\rVert_\infty t \int_{\R^d} \E^y\left[|K_{B_R^C}(x-X^0_t)|U^{b_\theta}_t\right]|w(y)|\rd y\\
         =:& I_1 +I_2+I_3+I_4
    \end{aligned}
\end{equation*}
For the first term $I_1$ on the right-hand side, we can apply the H{\"o}lder's  and Burkholder-Davis-Gundy inequalities and deduce that 
\begin{equation*}
    \begin{aligned}
        I_1 \leq & \left(\E^y \left[K_{B_R}^r(x-X^0_t)(U^{b_\theta}_t)^r\right]|w(y)|\rd y\right)^\frac{1}{r} \left(\int_{\R^d}\E^y[|Z_t|^s]|w(y)|\rd y\right)^\frac{1}{s}\\
        \leq &\left(\E^y \left[K_{B_R}^r(x-X^0_t)(U^{b_\theta}_t)^r\right]|w(y)|\rd y\right)^\frac{1}{r} \left(\int_{\R^d}\E^y[\langle Z_t\rangle^\frac{s}{2}]|w(y)|\rd y\right)^\frac{1}{s}.
    \end{aligned}
\end{equation*}
By the definition of $Z$ in \eqref{Z-martingale},
\begin{equation}
\label{Z-quadratic-variation}
    \begin{aligned}
        \langle Z\rangle_t = & \int^t_0 g_{ij}g_{kl} g^{jl} (b^i-\tb^i)(b^k-\tb^k)(s,X^0_s)\rd s\\
        = & \int^t_0 |b(s,X^0_s)-\tb(s,X^0_s)|_g^2\rd s\\
        \leq & \xi t \lVert b-\tb \rVert_\infty^2.
    \end{aligned}
\end{equation}
As for $(U^{b_\theta})^r$, we have that 
\begin{equation*}
    \begin{aligned}
        (U^{b_\theta}_t)^r = & \exp\left(r\int^t_0 \langle b_\theta(s,X^0_s),\rd M_s\rangle_g -\frac{r}{2}\int^t_0 |b_\theta(s,X^0_s)|_g^2 \rd s\right)\\
        =& \exp\left(\int^t_0 \langle rb_\theta(s,X^0_s),\rd M_s\rangle_g - \frac{1}{2}\int^t_0 |rb_\theta(s,X^0_s)|_g^2 \rd s+\frac{r(r-1)}{2}\int^t_0 |b_\theta(s,X^0_s)|_g^2\rd s\right)\\
        =& U^{rb_\theta}_t \cdot \exp\left(\frac{r(r-1)}{2}\int^t_0 |b_\theta(s,X^0_s)|_g^2\rd s\right)\\
        \leq & e^{\frac{r(r-1)}{2}\xi \lVert b_\theta\rVert_\infty ^2 t}U^{rb_\theta}_t,   
    \end{aligned}
\end{equation*}
where we have used that $\lVert b_\theta\rVert_\infty\leq \theta\lVert b\rVert_\infty +(1-\theta)\lVert \tb\rVert_\infty \leq L$. By definition, we also have that 
\begin{equation*}
    K_{B_R}^r(x-X^0_t)\leq \frac{\alpha^r}{|x-X^0_t|^{\gamma r}} 1_{\{|x-X^0_t|\leq R\}},
\end{equation*}
so we may conclude that 
\begin{equation*}
    \E^y \left[K_{B_R}^r(x-X^0_t)(U^{b_\theta}_t)^r\right] \leq \E^y \left[\frac{\alpha^r}{|x-X^0_t|^{\gamma r}} 1_{\{|x-X^0_t|\leq R\}} e^{\frac{r(r-1)}{2}\xi \lVert b_\theta\rVert_\infty^2 t}U^{rb_\theta}_t\right].
\end{equation*}
Together with \eqref{Z-quadratic-variation}, it follows that 
\begin{equation*}
\begin{aligned}
    I_1\leq & e^{\frac{r-1}{2}\xi \lVert b_\theta\rVert_\infty^2 t} \sqrt{\xi t}\lVert b-\tb\rVert_\infty 
    \left(\E^y \left[K_{B_R}^r(x-X^0_t)U^{rb_\theta}_t\right]|w(y)|\rd y\right)^\frac{1}{r} \left(\int_{\R^d}|w(y)|\rd y\right)^\frac{1}{s},
\end{aligned}
\end{equation*}
and from the proof of Lemma \ref{technical-lemma}, we can see that when we take $1<r<\frac{d}{\gamma}$, it follows that 
\begin{equation*}
    \begin{aligned}
        I_1\leq & e^{\frac{r-1}{2}\xi \lVert b_\theta\rVert_\infty^2 t} \sqrt{\xi t}\lVert b-\tb\rVert_\infty \lVert w\rVert_1^{\frac{1}{s}} \left(\int_{\R^d}\int_{B_R} \frac{\alpha^r}{|u|^{r\gamma}} p_{rb_\theta}(y,t,x-u)|w(y)|\rd u\rd y
        \right)^{\frac{1}{r}}\\
        \leq & \sqrt{\xi}\lVert b-\tb\rVert_\infty \lVert w\rVert_1^\frac{1}{s} e^{\frac{r-1}{2}\xi \lVert b_\theta\rVert_\infty^2 t}\sqrt{t}\\
        &\cdot \left(\frac{2\lVert w\rVert_\infty \pi^d}{\Gamma(\frac{d}{2})}C^{1+\frac{d}{2}} \alpha^r \frac{R^{d-r\gamma}}{d-r\gamma}\left(1+r\lVert b_\theta \rVert_\infty \sqrt{t} e^{\frac{\xi(p-1)}{2}r\lVert b_\theta\rVert_\infty^2 t}\right)\right)^\frac{1}{r}.
    \end{aligned}
\end{equation*}
Then for $t\leq T_L$, as $\lVert b_\theta\rVert_\infty\leq L$, $\lVert b_\theta \rVert_\infty \sqrt{t}\leq 1$, we deduce that 
\begin{equation*}
    I_1\leq \sqrt{\xi}\lVert b-\tb\rVert_\infty \lVert w\rVert_1^\frac{1}{s} e^{\frac{r-1}{2}\xi }\sqrt{t} \left(\frac{2\lVert w\rVert_\infty \pi^d}{\Gamma(\frac{d}{2})}C^{1+\frac{d}{2}} \alpha^r \frac{R^{d-r\gamma}}{d-r\gamma}\left(1+r e^{\frac{\xi(p-1)}{2}r}\right)\right)^\frac{1}{r}.
\end{equation*}
Similarly, for $I_2$, it follows from \eqref{Z-quadratic-variation} and the proof of Lemma \ref{technical-lemma} that 
\begin{equation*}
    \begin{aligned}
        I_2\leq & \left(\E^y \left[K_{B_R^C}^r(x-X^0_t)(U^{b_\theta}_t)^r\right]|w(y)|\rd y\right)^\frac{1}{r} \left(\int_{\R^d}\E^y[\langle Z_t\rangle^\frac{s}{2}]|w(y)|\rd y\right)^\frac{1}{s}\\
        \leq & \left(\int_{\R^d}\int_{B_R^C}\frac{\alpha^r}{|u|^{r\gamma}}p_{rb_\theta}(y,t,x-u)e^{\frac{r}{2}(r-1)\xi\lVert  b_\theta\rVert_\infty^2 t}\rd u |w(y)|\rd y\right)^{\frac{1}{r}}\\
        &\cdot \left(\int_{\R^d}(\xi t\lVert b-\tb\rVert_\infty^2)^\frac{s}{2}|w(y)|\rd y\right)^\frac{1}{s}\\
        \leq & \sqrt{\xi}\lVert b-\tb \rVert_\infty \lVert w\rVert_1^\frac{1}{s} e^{\frac{r-1}{2}\xi\lVert b_\theta\rVert_\infty^2 t} \left(\frac{\alpha^r}{R^{r\gamma}}\int_{\R^d}\int_{B_R^C} p_{rb_\theta}(y,t,x-u)\rd u |w(y)|\rd y\right)^\frac{1}{r}\\
        \leq & \sqrt{\xi}e^{\frac{r-1}{2}\xi\lVert b_\theta\rVert_\infty^2 t} \frac{\alpha}{R^\gamma}\lVert w\rVert_1 \lVert b-\tb \rVert_\infty \sqrt{t}.
    \end{aligned}
\end{equation*}
Again, when we take $t\leq T_L$, we arrive at 
\begin{equation*}
    I_2\leq \sqrt{\xi}e^{\frac{r-1}{2}\xi} \frac{\alpha}{R^\gamma}\lVert w\rVert_1 \lVert b-\tb \rVert_\infty \sqrt{t}.
\end{equation*}
Therefore, putting together the previous estimates,
\begin{equation*}
\begin{aligned}
    I_1+I_2\leq &\left( \lVert w\rVert_1^{1-\frac{1}{r}}  \left(\frac{2\lVert w\rVert_\infty \pi^d}{\Gamma(\frac{d}{2})}C^{1+\frac{d}{2}} \alpha^r \frac{R^{d-r\gamma}}{d-r\gamma}\left(1+r e^{\frac{\xi(p-1)}{2}r}\right)\right)^\frac{1}{r}+ \frac{\alpha}{R^\gamma}\lVert w\rVert_1  \right)\\
    &\cdot \sqrt{\xi}e^{\frac{r-1}{2}\xi}\lVert b-\tb\rVert_\infty\sqrt{t}.
\end{aligned}
\end{equation*}
As the above bound holds for all $r\in (1,\frac{d}{\gamma})$, so when we send $r\downarrow 1$, we obtain that 
\begin{equation*}
\begin{aligned}
    I_1 + I_2\leq &\lim_{r\downarrow 1} \left( \lVert w\rVert_1^{1-\frac{1}{r}}  \left(\frac{2\lVert w\rVert_\infty \pi^d}{\Gamma(\frac{d}{2})}C^{1+\frac{d}{2}} \alpha^r \frac{R^{d-r\gamma}}{d-r\gamma}\left(1+r e^{\frac{\xi(p-1)}{2}r}\right)\right)^\frac{1}{r}+ \frac{\alpha}{R^\gamma}\lVert w\rVert_1  \right)\\
    &\cdot \sqrt{\xi}e^{\frac{r-1}{2}\xi}\lVert b-\tb\rVert_\infty\sqrt{t}\\
    = & \left(\frac{2 \alpha C^{1+\frac{d}{2}} \pi^d R^{d-\gamma} \lVert w\rVert_\infty }{\Gamma(\frac{d}{2})(d-\gamma)}\left(1+ e^{\frac{\xi(p-1)}{2}}\right)+ \frac{\alpha \lVert w\rVert_1}{R^\gamma}  \right) \sqrt{\xi}\lVert b-\tb\rVert_\infty\sqrt{t}\\
    = & C_0 \sqrt{\xi}\lVert b-\tb\rVert_\infty\sqrt{t}. 
\end{aligned}
\end{equation*}
Next let us handle the rest two terms $I_3$ and $I_4$. Again, from the proof of Lemma \ref{technical-lemma}, we see that 
\begin{equation*}
    \begin{aligned}
        I_3 =&  \xi L \lVert b-\tb\rVert_\infty t \int_{\R^d} \E^y\left[|K_{B_R}(x-X^0_t)|U^{b_\theta}_t\right]|w(y)|\rd y\\
        \leq &\xi L \lVert b-\tb\rVert_\infty t \int_{\R^d} \int_{B_R} \frac{\alpha}{|u|^\gamma}p_{b_\theta}(y,t,x-u)\rd u |w(y)|\rd y\\
        \leq & \xi L \lVert b-\tb\rVert_\infty t \frac{2\lVert w\rVert_\infty \pi^d}{\Gamma(\frac{d}{2})}C^{1+\frac{d}{2}}\alpha \left(1+ \lVert b_\theta\rVert_\infty \sqrt{t}e^{\frac{\xi(p-1)}{2}\lVert b_\theta \rVert_\infty^2 t}\right)\frac{R^{d-\gamma}}{d-\gamma},
    \end{aligned}
\end{equation*}
and under the assumption that $t\leq T_L$, $\lVert b_\theta\rVert_\infty \sqrt{t}\leq L\sqrt{t}\leq 1$, so it follows that
\begin{equation*}
    I_3 \leq \frac{2\lVert w\rVert_\infty \pi^d}{\Gamma(\frac{d}{2})}C^{1+\frac{d}{2}}\alpha \left(1+ e^{\frac{\xi(p-1)}{2}}\right)\frac{R^{d-\gamma}}{d-\gamma} \xi \lVert b-\tb\rVert_\infty Lt.  
\end{equation*}
Finally, for $I_4$, we have that when $t\leq T_L$,
\begin{equation*}
    \begin{aligned}
        I_4 \leq  & \xi L \lVert b-\tb\rVert_\infty t \int_{\R^d} \int_{B_R^C} \frac{\alpha}{|u|^\gamma} p_{b_\theta}(y,t,x-u)\rd u|w(y)|\rd y \\
        \leq & \xi L  \lVert b-\tb\rVert_\infty t \frac{\alpha}{R^\gamma}\lVert w\rVert_1\\
        \leq & \xi \lVert b-\tb\rVert_\infty Lt \frac{\alpha}{R^\gamma}\lVert w\rVert_1.
    \end{aligned}
\end{equation*}
Therefore, combining the above results, we conclude that 
\begin{equation*}
    \begin{aligned}
        I_3+I_4 \leq & \left(\frac{2\alpha C^{1+\frac{d}{2}}\pi^d R^{d-\gamma}\lVert w\rVert_\infty }{\Gamma(\frac{d}{2})(d-\gamma)} \left(1+ e^{\frac{\xi(p-1)}{2}}\right) + \frac{\alpha \lVert w\rVert_1}{R^\gamma}\right) \xi \lVert b-\tb\rVert_\infty Lt\\
        =& C_0 \xi \lVert b-\tb\rVert_\infty Lt\\
        \leq & C_0 \xi \lVert b-\tb\rVert_\infty \sqrt{t}. 
    \end{aligned}
\end{equation*}
As a consequence, 
\begin{equation*}
\begin{aligned}
    \left|\K(b)(t,x)-\K(\tb)(t,x)\right|\leq C_0(\xi+\sqrt{\xi})\sqrt{t}\lVert b-\tb\rVert_\infty.
\end{aligned}
\end{equation*}
When we take $t\leq \tau<(1\wedge \frac{1}{\xi+\sqrt{\xi}})T_L$,
\begin{equation*}
    \lVert \K(b) -\K(\tb)\rVert_{L^\infty([0,\tau]\times \R^d)}\leq C_0(\xi+\sqrt{\xi})\sqrt{\tau}\lVert b-\tb\rVert_\infty, 
\end{equation*}
where $C_0(\xi+\sqrt{\xi})\sqrt{\tau}<1$, and thus there exists a unique fixed point $b\in \mathcal{B}_L([0,\tau]\times \R^d)$ such that $\K(b)=b$. The result follows from the classical result on the weak solutions to the SDE \eqref{SDE-b-drift}.
\end{proof}

\bibliographystyle{abbrv}
\nocite*
\bibliography{refs}
\end{document}